\documentclass{amsart}
\usepackage{amssymb,amsmath,latexsym}
\usepackage{amsthm}
\usepackage{fontenc}
\usepackage{mathrsfs}
\usepackage{amssymb}

\numberwithin{equation}{section}

\newtheorem{theorem}{Theorem}[section]

\begin{document}
\author{Alexander E. Patkowski}
\title{On Berndt's summation formula}

\maketitle
\begin{abstract} We offer a proof of a summation formula equivalent to one due to Berndt. Our proof uses the M$\ddot{u}$ntz formula and the Poisson summation formula. By utilizing known properties of Mellin inversion, we give an example from a discontinuous function. Several new applications are offered as corollaries. \end{abstract}



\section{Introduction and Main Summation formulas} 
In [2], Berndt offered an intriguing summation formula involving arithmetic functions. Let $a(n)$ be an arithmetic function in the sense that its domain is $\mathbb{N}$ and range is $\mathbb{C},$ and $b(n):=\sum_{d|n}a(d)\mu(n/d).$ As usual $\mu(n)$ is the M$\ddot{o}$bius function. We also define $a(n)=a(-n)$ for each positive integer $n.$ Define $a(n,S):=\sum_{\substack{d|n \\ d\in S}}b(d),$ if $S\subset\mathbb{N}.$ Suppose that $f(x)$ is continuous on the real line $\mathbb{R}$ and $f(x)\in L_1(\mathbb{R}).$ Then, assuming $$\sum_{\substack{n\in\mathbb{Z}\\n\neq0}}a(n)f(n)$$ converges absolutely, Berndt's formula is stated as [2, pg.295--296, eq.(3)]
\begin{equation}\sum_{\substack{n\in\mathbb{Z}\\n\neq0}}a(n,S)f(n)=\sum_{\substack{k\ge1 \\ k\in S}}\frac{b(k)}{k}\left(\sum_{m\in\mathbb{Z}}\int_{\mathbb{R}}e^{2\pi imx/k}f(x)dx-f(0)k \right). \end{equation}
Berndt't proof of this formula involves constructing a special function from arithmetic functions and applying the "ordinary" Poisson summation formula. We were able to find a similar summation formula using some known results on Mellin transforms. Recall that the Mellin transform of a function $f(x)$ is given as 
\begin{equation}\mathfrak{M}(f)(s)=\int_{0}^{\infty}x^{s-1}f(x)dx,\end{equation}
for $a<\Re(s)<b$ provided $f(x)$ satisfies the growth condition $O(x^{-a-\epsilon}),$ as $x\rightarrow0^{+}$ and $O(x^{-b+\epsilon}),$ as $x\rightarrow\infty$ [12, pg.80]. Additionally, the inverse formula is given as
\begin{equation}\mathfrak{M}^{-1}\left(\mathfrak{M}(f)(s)\right)(x)=f(x)=\frac{1}{2\pi i}\int_{c-i\infty}^{c+i\infty}x^{-s}\mathfrak{M}(f)(s)ds,\end{equation}
provided $\Re(s)=c$ is where $\mathfrak{M}(f)(s)$ is taken to be analytic.  Let $\mathscr{F}(f)(w)=\int_{0}^{\infty}\cos(w2\pi x)f(x)dx$ denote the Fourier cosine transform of $f.$ The summation formula we offer is constructed from the Poisson summation formula for Fourier cosine transforms [3, pg.233, eq.(31.2), $a=0,$ $b=\infty$]
$$\frac{f(0)}{2}+\sum_{n\ge1}f(n)=\int_{0}^{\infty}f(x)dx+2\sum_{n\ge1}\mathscr{F}(f)(n).$$
\begin{theorem} Assume $f$ satisfies the hypothesis of (1.1) and growth conditions for (1.2). Let $\zeta(s)$ be the Riemann zeta function. Then, assuming absolute convergence,
$$\sum_{n\ge1}a(n)\mathscr{F}(f)(n)=\frac{1}{2}\sum_{m\ge1}\frac{b(m)}{m}\left(\mathfrak{M}^{-1}\left(\zeta(s) \mathfrak{M}(f)(s)\right)(1/m)+\frac{f(0)}{2}\right),$$
where $\mathfrak{M}^{-1}$ is taken over the vertical line $0<\Re(s)=c<1.$
\end{theorem}
\begin{proof} We write
$$\begin{aligned}\sum_{n\ge1}a(n)\mathscr{F}(f)(n)&=\sum_{m\ge1}b(m)\sum_{n\ge1}\int_{0}^{\infty}\cos(2nm\pi x)f(x)dx
\\ &=\sum_{m\ge1}\frac{b(m)}{m}\sum_{n\ge1}\int_{0}^{\infty}\cos(2n\pi x)f(\frac{x}{m})dx\\
&=\frac{1}{2}\sum_{m\ge1}\frac{b(m)}{m}\left(-\int_{0}^{\infty}f\left(\frac{x}{m}\right)dx+\frac{f(0)}{2}+\sum_{n\ge1}f(\frac{n}{m}) \right)\\
&=\frac{1}{2}\sum_{m\ge1}\frac{b(m)}{m}\left(\mathfrak{M}^{-1}\left(\zeta(s) \mathfrak{M}(f)(s)\right)(1/m)+\frac{f(0)}{2}\right).\end{aligned}$$

In the second line we have made the change of variable $x\rightarrow x/m.$ In the third line we invoked the Poisson summation formula for cosine transforms. In the fourth line we have invoked the M$\ddot{u}$ntz formula [14, pg.29, eq.(2.11.1), $x\rightarrow 1/x,$ $v\rightarrow v/x$]
\begin{equation}\mathfrak{M}^{-1}\left(\zeta(s) \mathfrak{M}(f)(s)\right)(1/x)=\frac{1}{2\pi i}\int_{c-i\infty}^{c+i\infty}x^s\zeta(s) \mathfrak{M}(f)(s)ds=-\int_{0}^{\infty}f\left(\frac{y}{x}\right)dy+\sum_{n\ge1}f\left(\frac{n}{x}\right), \end{equation}
where $0<\Re(s)=c<1.$ This formula may be applied, since the first order derivative of $f$ is continuous, $f(x)$ has growth $O(x^{-m}),$ $m>1,$  for large $x$ by hypothesis of the theorem. 
\end{proof}
Theorem 1.1 may be recovered from Berndt's formula (1.1) by choosing the set $S=\mathbb{N},$ and the function $f$ to be the Fourier cosine transform of a function. Since $f$ is then an even function, the result follows after proceeding with our steps involving the M$\ddot{u}$ntz formula. One advantage to our application of Mellin inversion, is we may work with discontinuous functions to obtain formulas valid almost everywhere. An example of this will be provided in the next theorem. As a result, we are able to produce applications from Theorem 1.1 not included in the list provided in [2], which restricted examples to continuous functions. \par H. Davenport [4] offered a curious Fourier series for a series involving the fractional part function over a class of arithmetic functions. Let $[x]$ be the integer part of $x,$ and write $\{x\}=x-[x],$ Note that,
$$[x] := \begin{cases}\lfloor{x\rfloor},& \text {if } x\ge0,\\ \lceil{x\rceil}, & \text{if } x<0.\end{cases}$$
For a well-known proof involving Mellin inversion see [13]. Although Theorem 1.1 assumes continuity in the hypothesis, we may extend our examples to discontinuous functions by exploiting properties of Mellin inversion. It is a known property that Mellin inversion recovers the original function in this case almost everywhere when discontinuities exist (e.g. [12, pg.93]). The identity we wish to prove in the next theorem has also been noted in the theorem in [13, pg.348] to assume absolute  and uniform convergence of the Dirichlet series $\sum_{n\ge1}b(n)n^{-s},$ for $\Re(s)>0.$
\begin{theorem} ([4]) For irrational $x>0,$ and assuming absolute and uniform convergence of the Dirichlet series $\sum_{n\ge1}b(n)n^{-s},$ for $\Re(s)=1+\delta,$ $\delta>0,$
\begin{equation}\sum_{n\ge1}\frac{b(n)}{n}\left(\{nx\}-\frac{1}{2}\right)=-\frac{1}{\pi}\sum_{n\ge1}\frac{a(n)}{n}\sin(2\pi n x).\end{equation}
\end{theorem}
\begin{proof} First, we select $$f(w)=\int_{0}^{\infty}\cos(y2\pi w)\frac{\sin(2\pi xy)}{y}dy$$ in Theorem 1.1 and notice that by [12]
\begin{equation}\int_{0}^{\infty}w^{s-1}\left(\int_{0}^{\infty}\cos(y2\pi w)\frac{\sin(2\pi xy)}{y}dy \right)dw=\Gamma(s)\cos(\frac{\pi}{2}s)\Gamma(-s)\sin(\frac{\pi}{2}s)x^{s}=\frac{\pi x^{s}}{2s},\end{equation}
where we have used the reflection formula for the gamma function, provided $0<\Re(s)<1.$ It is known from Perron's formula [14, pg.14, eq.(2.1.5)] that, for $0<c<1,$ $x>0,$
$$-\frac{1}{2\pi i}\int_{c-i\infty}^{c+i\infty}\zeta(s)\frac{x^{s}}{s}ds=\{x\}.$$
Now using $f(0)=\frac{\pi}{2}$ we have the result after applying Fourier inversion on the left hand side of Theorem 1.1, and noting that $\{x\}$ is continuous only at the irrational numbers.
\end{proof}
Our proof may be conveniently applied to other summation formula's of the Voronoi type. Indeed, let $\sigma(n)$ denote the number of divisors of $n.$ Define $\mathscr{K}(f)(x)=\int_{0}^{\infty}f(y)(4K_0(4\pi\sqrt{xy})-2\pi Y_0(4\pi\sqrt{xy}))dy,$ where $K_0(x)$ and $Y_0(x)$ are the modified Bessel functions. The Voronoi's summation formula for $\sigma(n)$ is known to be [1, pg.139]
$$\frac{f(0)}{2}+\sum_{n\ge1}\sigma(n)f(n)=\int_{0}^{\infty}f(x)(\log(x)+2\gamma)dx+\sum_{n\ge1}\sigma(n)\mathscr{K}(f)(n).$$
Here $\gamma$ is the Euler-Mascheroni constant.
\begin{theorem} Let $c(n)=\sum_{d|n}\sigma(\frac{n}{d})b(d).$ Assume $f$ satisfies the hypothesis of (1.1) and growth conditions for (1.2). Then, assuming absolute convergence,
$$\sum_{n\ge1}c(n)\mathscr{K}(f)(n)=\sum_{m\ge1}\frac{b(m)}{m}\left(\mathfrak{M}^{-1}\left(\zeta^2(s) \mathfrak{M}(f)(s)\right)(1/m)+\frac{f(0)}{2}\right),$$
where $\mathfrak{M}^{-1}$ is taken over the vertical line $0<\Re(s)=c<1.$
\end{theorem}
\begin{proof} The proof is identical to the one for Theorem 1.1, but requires a M$\ddot{u}$ntz-type formula for $\sigma(n)$ that recently appeared in [10, pg.404, Theorem 3.4]
\begin{equation}\frac{1}{2\pi i}\int_{c-i\infty}^{c+i\infty}x^s\zeta^2(s) \mathfrak{M}(f)(s)ds=-\int_{0}^{\infty}f\left(\frac{y}{x}\right)(\log(y)+2\gamma)dy+\sum_{n\ge1}\sigma(n)f\left(\frac{n}{x}\right), \end{equation}
where $0<\Re(s)=c<1.$ The result follows after splitting the sum over $c(n)$ into two series, and making a change of variable as before. \end{proof}

\section{Identity involving the Koshlyakov function}
The function
\begin{equation}\mathfrak{K}(x):=2\sum_{n\ge1}\sigma(n)\left(K_0(4\pi e^{i\pi/4}\sqrt{nx})+K_0(4\pi e^{-i\pi/4}\sqrt{nx})\right) \end{equation}
where $\sigma(n)$ denotes the number of divisors of $n$ and $K_0(x)$ is the modified Bessel function of the second kind, has recently appeared in several papers, and is attributed in [5] to Koshlyakov. We give an apparently new formula involving this function by applying Theorem 1.1. 
\begin{theorem} For real $z>0,$ assuming $a(n)$ is chosen so that the series converge absolutely,
$$\sum_{n\ge1}a(n)I(n,z)=\frac{z^2}{2}\sum_{m\ge1}\frac{b(m)}{m}\left(zm\mathfrak{K}(zm)+\frac{1}{2\pi}-\frac{1}{8z^2}\right),$$
where $$I(x,z)=\int_{0}^{\infty}\cos(xw2\pi)w^{-2}\int_{0}^{\infty}e^{-(zy)^2/w^2}y\left(\sum_{n\ge1}e^{-(yn)^2}-\frac{\sqrt{\pi}}{2y}\right)dydw.$$
\end{theorem}

\begin{proof} We choose the function as the absolutely convergent integral $$f(x)=x^{-2}\int_{0}^{\infty}e^{-(zy)^2/x^2}y\left(\sum_{n\ge1}e^{-(yn)^2}-\frac{\sqrt{\pi}}{2y}\right)dy,$$ in Theorem 1.1. The left side of Theorem 2.1 is clear. Note that for $0<\Re(s)<1,$ 

\begin{equation}\begin{aligned}\mathfrak{M}(f)(s)&=\int_{0}^{\infty}x^{s-1}x^{-2} \int_{0}^{\infty}e^{-(zy)^2/x^2}y\left(\sum_{n\ge1}e^{-(yn)^2}-\frac{\sqrt{\pi}}{2y}\right)dy dx\\
&=\Gamma(1-\frac{s}{2})z^{s-2}\int_{0}^{\infty}y^{s-1}\left(\sum_{n\ge1}e^{-(yn)^2}-\frac{\sqrt{\pi}}{2y}\right)dy\\
&=\Gamma(\frac{s}{2})\Gamma(1-\frac{s}{2})z^{s-2}\zeta(s)=\frac{z^{s-2}\pi\zeta(s)}{\sin(\frac{\pi}{2}s)}. \end{aligned}\end{equation}
In the last line we invoked a well-known Mellin transform which was nicely discussed in Ivi\'c's paper [7], as well as the reflection formula for the gamma function. We require an integral evaluation from [5, pg.243, eq.(6.5)--(6.6), $z=0$],[9, eq.(11)], $c>1,$
\begin{equation}\mathfrak{K}(x)=\frac{1}{2\pi i}\int_{c-i\infty}^{c+i\infty}\frac{\zeta^2(1-s)x^{-s}}{2\cos(\frac{\pi}{2}s)}ds. \end{equation}
The integrand may be seen to have exponential decay from Stirling's formula for the gamma function once employing $\zeta(1-s)=2(2\pi)^{-s}\cos(\frac{\pi}{2}s)\Gamma(s)\zeta(s)$ [14, pg.16, eq.(2.1.8)]. Hence, we may move the line of integration to $0<d<1$ and compute the residue at the simple pole $s=1,$ to obtain for $0<d<1$
\begin{equation}\mathfrak{K}(x)=-\frac{1}{2x\pi}+\frac{1}{2\pi i}\int_{d-i\infty}^{d+i\infty}\frac{\zeta^2(1-s)x^{-s}}{2\cos(\frac{\pi}{2}s)}ds.\end{equation}
The residue is easily obtained by differentiating the denominator of the integrand when setting $s=1$ and using $\zeta(0)=-\frac{1}{2}.$ Now replacing $s$ by $1-s$ in this last integral and noting $\sin(\frac{\pi}{2}(1-s))=\cos(\frac{\pi}{2}s),$ we obtain
\begin{equation}x\mathfrak{K}(x)+\frac{1}{2\pi}=\frac{1}{2\pi i}\int_{d-i\infty}^{d+i\infty}\frac{\zeta^2(s)x^{s}}{2\sin(\frac{\pi}{2}s)}ds.\end{equation}
Through some properties of the absolutely convergent integral for $f(x),$ it can be seen that $f(0)=-\frac{1}{4z^2}.$ To see this we make a change of variables in our integral for $f(x)$ to eliminate the prefactor $x^{-2},$ invoke the theta function transformation [12, pg.120, eq.(4.1.12)], and take the limit as $x\rightarrow0$ in the result. The computation is,
\begin{align*} f(x)&=x^{-2}\int_{0}^{\infty}e^{-(zy)^2/x^2}y\left(\sum_{n\ge1}e^{-(yn)^2}-\frac{\sqrt{\pi}}{2y}\right)dy\\
&=x^{-2}\int_{0}^{\infty}e^{-(zy)^2/x^2}\left(\sqrt{\pi}\sum_{n\ge1}e^{-(\pi n/y)^2}-\frac{y}{2}\right)dy\\
&=x^{-2}\sqrt{\pi}\sum_{n\ge1}\int_{0}^{\infty}e^{-(zy)^2/x^2}e^{-(\pi n/y)^2}dy-x^{-2}\int_{0}^{\infty}e^{-(zy)^2/x^2}\frac{y}{2}dy\\
&=\frac{\sqrt{\pi}}{2x}\sum_{n\ge1}\int_{0}^{\infty}e^{-y^2}e^{-(\pi zn/y)^2}dy-x^{-2}\int_{0}^{\infty}e^{-(zy)^2/x^2}\frac{y}{2}dy\\
&=\frac{\pi}{2xz}\frac{1}{e^{2\pi z/x}-1}-\frac{1}{4z^2}.
\end{align*}
The last line follows from the integral evaluation [6, pg.146, eq.(27), $p=1$]. The first term clearly tends to $0$ as $x\rightarrow0.$ Now combining (2.2) with Theorem 1.1 gives the right side of Theorem 2.1 upon invoking (2.5). 
\end{proof}
\section{Rearranging Motohashi's formula}
An intriguing formula for computing integrals involving the mean square of the Riemann zeta function on the critical line was given by Motohashi in [11, Theorem 4.1] (see also Ivi\'c's paper [8]). It is stated as
$$\int_{-\infty}^{\infty}f(y)\left|\zeta\left(\frac{1}{2}+iy\right)\right|^2dy=\int_{-\infty}^{\infty}f(y)\left(\psi(\frac{1}{2}+iy)-i\frac{\pi}{2}\tanh(\pi y) \right)dy$$
\begin{equation}+2\pi\Re(f(\frac{i}{2}))+4\sum_{n\ge1}\sigma(n)\int_{0}^{\infty}(y(y+1))^{-1/2}\mathscr{F}(f)(\log(1+1/y))\cos(2\pi ny)dy.\end{equation}
The main difficulty in applying the formula appears to be in evaluating the last integral in the series on the right side of (3.1). We were able to adapt our method to recasting (3.1) in a different form. Define $h(y)=(y(y+1))^{-1/2}\mathscr{F}(f)(\log(1+1/y)).$ Notice that by Theorem 1.1 with $a(n)=\sigma(n),$ 
\begin{equation}\begin{aligned}&\sum_{n\ge1}\sigma(n)\int_{0}^{\infty}(y(y+1))^{-1/2}\mathscr{F}(f)(\log(1+1/y))\cos(2\pi ny)dy\\ &=\sum_{n\ge1}\sum_{m\ge1}\frac{1}{m}\int_{0}^{\infty}(y/m(y/m+1))^{-1/2}\mathscr{F}(f)(\log(1+m/y))\cos(2\pi ny)dy \\
&=\frac{1}{2}\sum_{m\ge1}\frac{1}{m}\left(\mathfrak{M}^{-1}\left(\zeta(s) \mathfrak{M}(h)(s)\right)(1/m)+\frac{h(0)}{2}\right)  \end{aligned}\end{equation}
assuming $h(0)$ exists. We may now write,
 \begin{equation}\begin{aligned}\mathfrak{M}(h)(s)&=\int_{0}^{\infty}y^{s-1}(y(y+1))^{-1/2}\int_{0}^{\infty}f(x)\cos(2\pi x\log(1+1/y))dxdy\\ 
&=\frac{1}{2}\int_{0}^{\infty}y^{s-1}(y(y+1))^{-1/2}\int_{0}^{\infty}f(x)((1+1/y)^{i2\pi x}+(1+1/y)^{-i2\pi x})dxdy\\
&=\frac{1}{2}\int_{0}^{\infty}f(x)\left(\frac{\Gamma(s-i2\pi x-\frac{1}{2})\Gamma(1-s)}{\Gamma(\frac{1}{2}-i2\pi x)}+\frac{\Gamma(s+i2\pi x-\frac{1}{2})\Gamma(1-s)}{\Gamma(\frac{1}{2}+i2\pi x)}\right)dx, \end{aligned}\end{equation}
for $\frac{1}{2}<\Re(s)<1,$ by [6, pg.310, eq.(19)]
\begin{equation}\int_{0}^{\infty}\frac{y^{s-1}}{(1+y)^{v}}dy=\frac{\Gamma(s)\Gamma(v-s)}{\Gamma(v)}, \end{equation}
for $0<\Re(s)<\Re(v).$ Parseval's formula [12, pg.83, eq.(3.1.11)] tells us that for $\Re(a-\frac{1}{2})<c<1,$
\begin{equation} \int_{0}^{\infty}\left(\sum_{n\ge1}e^{-nyx}n^{a-\frac{1}{2}}y^{a-\frac{1}{2}}-\frac{\Gamma(a+\frac{1}{2})}{xy}\right)e^{-y}dy=\frac{1}{2\pi i}\int_{c-i\infty}^{c+i\infty}x^{-s}\zeta(s)\Gamma(s+a-\frac{1}{2})\Gamma(1-s)ds.\end{equation}
Therefore, assuming absolute convergence from our choice of $f,$ taking the inverse Mellin transform of (3.3) gives us
\begin{equation}\mathfrak{M}^{-1}\left(\zeta(s) \mathfrak{M}(h)(s)\right)(1/m)=\frac{1}{2}\int_{0}^{\infty}f(x)(G(i2\pi x,1/m)+G(-i2\pi x,1/m))dx, \end{equation}
by (3.5) with $a=\pm i2\pi x,$ where
$$G(a,x):=\frac{1}{\Gamma(\frac{1}{2}-a)}\int_{0}^{\infty}\left(\sum_{n\ge1}e^{-nyx}n^{a-\frac{1}{2}}y^{a-\frac{1}{2}}-\frac{\Gamma(a+\frac{1}{2})}{xy}\right)e^{-y}dy.$$ We remark that due to Stirling's formula for the gamma function, the integrand of the last line in (3.3) behaves like $f$ for large $x.$ This allows for a large class of functions for which the interchange of integration is justified. Consequently, assuming absolute convergence, we have shown that the series on the far right hand side of (3.1) has the form
$$2\sum_{m\ge1}\frac{1}{m}\left(\int_{0}^{\infty}f(x)(G(i2\pi x,1/m)+G(-i2\pi x,1/m))dx +\frac{h(0)}{2}\right).$$

1390 Bumps River Rd. \\*
Centerville, MA
02632 \\*
USA \\*
ul. A. E. Ody\'{n}ca 47 \\*
02-606 Warsaw\\*
Poland\\*
E-mail: alexpatk@hotmail.com, alexepatkowski@gmail.com
\\*
Competing interests: The author declares none.
\\*
Funding statement: The author did not receive support for the submitted work.
\end{document}